\documentclass[12pt]{article}
\usepackage{amsmath,amssymb,amsthm,url}
\usepackage{enumerate}
\usepackage{graphicx}
\usepackage[dvips]{color}
\newtheorem{theorem}{Theorem}[section]

\newtheorem{definition}[theorem]{Definition}
\newtheorem{lemma}[theorem]{Lemma}
\newtheorem{proposition}[theorem]{Proposition}

\newtheorem{conjecture}[theorem]{Conjecture}

\begin{document}

\title{Relating the total domination number and the annihilation number for quasi-trees and some composite graphs}

\author{
Hongbo Hua $^{a,}$\thanks{Correspondence should be addressed to:
 hongbo\_hua@163.com (H. Hua), xyhuamath@163.com (X. Hua),
 sandi.klavzar@fmf.uni-lj.si (S. Klav\v zar), kexxu1221@126.com (K. Xu)}
\and
Xinying Hua $^{b,c}$
\and
Sandi Klav\v zar $^{d,e,f}$
\and
Kexiang Xu $^{b,c}$
}
\maketitle

\begin{center}
	$^a$ Faculty of Mathematics and Physics, Huaiyin Institute of Technology,
 Huai'an, Jiangsu 223003, PR China \\
	\medskip

	$^b$ College of Mathematics, Nanjing University of Aeronautics \& Astronautics, Nanjing, Jiangsu 210016, PR China\\
	\medskip

   $^c$ MIIT Key Laboratory of Mathematical Modelling and High Performance
 Computing of Air Vehicles, Nanjing 210016, China

	$^d$ Faculty of Mathematics and Physics, University of Ljubljana, Slovenia\\
	\medskip
		
	$^e$ Institute of Mathematics, Physics and Mechanics, Ljubljana, Slovenia\\
	\medskip
	
	$^f$ Faculty of Natural Sciences and Mathematics, University of Maribor, Slovenia\\
	\medskip
\end{center}

\begin{abstract}
The total domination number $\gamma_{t}(G)$ of a graph $G$ is the cardinality of a smallest set $D\subseteq V(G)$ such that each vertex of $G$  has a neighbor in $D$. The annihilation number $a(G)$ of $G$ is the largest integer $k$ such that there exist $k$ different vertices in $G$ with the degree sum at most $m(G)$. It is conjectured that $\gamma_{t}(G)\leq a(G)+1$ holds for every nontrivial connected graph $G$. The conjecture has been proved for graphs with minimum degree at least $3$, trees, certain tree-like graphs, block graphs, and cactus graphs. In the main result of this paper it is proved that the conjecture holds for quasi-trees.  The conjecture is verified also for some graph   constructions including bijection graphs,  Mycielskians, and the newly introduced universally-identifying graphs.
\end{abstract}

\noindent
{\bf Keywords:} total domination number; annihilation number; quasi-trees; bijection graph; Mycielskian \\

\noindent
AMS Subj.\ Class.\ (2020): 05C69, 05C07

\baselineskip20pt

\section{Introduction}

Let $G = (V(G), E(G))$ be a simple graph. If $v\in V(G)$, then $N_{G}(v)$ is its neighborhood and $d_{G}(v) = |N_{G}(v)|$ its degree. $D\subseteq V(G)$ is a {\em total dominating set} if each vertex from $V(G)$ has a neighbor in $D$. The minimum cardinality of a total dominating set is the \emph{total domination number}, $\gamma_{t}(G)$, of $G$. For an in-depth information on the total domination see the excellent book~\cite{Henning}.

Let $d_{1}\leq d_{2}\leq \cdots\leq d_{n}$ be the (ordered) degree sequence of a graph $G$. The \emph{annihilation number}, $a(G)$, of $G$ is the largest integer $k$ such that  $\sum\limits_{i=1}^{k}d_{i}\leq m(G)$, where $m(G)$ is the size of $G$. (The order of $G$ will be denoted by $n(G)$.) In other words, the annihilation number is the largest integer $k$ such that
$\sum\limits_{i=1}^{k}d_{i}\leq\sum\limits_{i=k+1}^{n}d_{i}$. This concept  was introduced by Pepper in~\cite{pepper-2004}. In this paper we are interested in the following conjecture posed in a somewhat different form in Graffiti.pc~\cite{DeLaVina} and later reformulated by Desormeaux, Haynes, and Henning in~\cite[Question 1]{Desormeaux}. (To be historically correct we emphasize that in~\cite{Desormeaux} this conjecture was posed as a problem, but nowadays it is called a conjecture and we follow this naming.)

\begin{conjecture}[\cite{DeLaVina,Desormeaux}]
\label{con1}
If $G$ is a connected graph with $n(G)\ge 2$, then $\gamma_{t}(G)\leq a(G)+1$.
\end{conjecture}

It immediately follows from the definition of the annihilation number that $a(G) \geq \lfloor \frac{n(G)}{2}\rfloor$. Since it was proved in~\cite{archdeacon-2004} that if the minimum degree $\delta(G)$ of $G$ is at least $3$, then $\gamma_{t}(G) \le \lfloor \frac{n(G)}{2}\rfloor$ (see also~\cite{Thomasse} for this result and~\cite{henning-2007} for its generalization),  Conjecture~\ref{con1} holds for graphs $G$ with  minimum degree $\delta(G)\ge 3$. In the seminal paper~\cite{Desormeaux}, the conjecture was verified for trees, while recently Bujt\'{a}s and Jakovac verified it for cactus graphs and for block graphs~\cite{Bujtas}. In~\cite{Yue}, the conjecture was further verified for the so-called $C$-disjoint graphs and for generalized theta graphs.
A graph $G$ is a \emph{quasi-tree} if there exists a vertex $x\in V(G)$ such that $G-x$ is a tree. Clearly, any tree is also a quasi-tree since it remains to be a tree after any leaf in it is removed. We say that a quasi-tree $G$ is \textit{non-trivial} if $G$ is not a tree. Since Conjecture~\ref{con1} holds for all trees, we only consider non-trivial quasi-trees throughout this paper. The main result of this paper reads as follows.

\begin{theorem}\label{t3}
If $G$ is a non-trivial quasi-tree, then $\gamma_{t}(G)\leq a(G)+1$.
\end{theorem}

We pose an open problem to characterize the quasi-trees $G$ satisfying the equality $\gamma_{t}(G)=a(G)+1$. We add that a conjecture parallel to~\ref{con1} has been posed also for the $2$-domination number of a graph $G$. In~\cite{desor-2014} (see also~\cite{lyle-2017}), the latter conjecture was verified for trees, and in~\cite{jakovac-2019} for block graphs. In~\cite{Yue-2020-2}, the conjecture was disproved by demonstrating that the $2$-domination number can be arbitrarily larger than the annihilation number. However, the counterexamples presented are far from being counterexamples for Conjecture~\ref{con1} and the authors say that they are ``inclined to believe that Conjecture~\ref{con1} holds true." The annihilation number was compared with the Roman domination number in~\cite{aram-2018} and with the locating-total domination number in~\cite{ning-2019}.

We proceed as follows. In the rest of this section we list some further definitions needed. In Section~\ref{sec:proof}, a proof of Theorem~\ref{t3} is given, while in the final section we confirm the validity of Conjecture~\ref{con1} for several graph operations which also generate graphs that have vertices of degree at most $2$.

Let $G$ be a graph. Then $S\subseteq V(G)$ is an \emph{annihilation set} if  $\sum\limits_{v\in S} d_{G}(v)\leq m(G)$. $S$ is an \emph{optimal annihilation set} if $|S|=a(G)$ and $\max\{d_{G}(v)|v\in S\} \le \min\{d_{G}(u)|u\in V(G)\setminus S\}$. A total dominating set of cardinality $\gamma_{t}(G)$ is called a {\em $\gamma_t$-set} of $G$. By $G[S]$ we denote the subgraph of $G$ induced by all vertices in $S\subseteq V(G)$. For a subset $S \subseteq V(G)$, we define $\sum(S,G)=\sum_{v\in S}d_G(v)$. The path and the cycle of order $n$ are respectively denoted by $P_{n}$ and $C_{n}$.  A \emph{generalized theta graph} $\Theta_{s_{1},\,\ldots,\,s_{k}}$ is formed by taking a pair of vertices $u, v$ and joining them by $k$ internally disjoint paths of lengths $s_{1},\,\ldots,\,s_{k}$, where $k\geq 3$. In particular, $\Theta_{s_{1},\,s_{2},\,s_{3}}$ is said to be a \emph{theta graph}. Let $C$ be the set containing all cycles, cliques and generalized theta graphs. We say a connected graph $F$ is \emph{$C$-disjoint} if any two subgraphs from $C$ in $F$ have no edge in common. $C$-disjoint graphs form a natural generalization of trees, cactus graphs, and block graphs. From our perspective, the most important thing is that Yue, Zhu, and Wei~\cite{Yue} proved Conjecture~\ref{con1} for all $C$-disjoint graphs.

For further graph terminology and notation not defined here see~\cite{west-2001}. Finally, for a positive integer $n$ we use the convention $[n] = \{1,\ldots, n\}$.

\section{Proof of Theorem~\ref{t3}}
\label{sec:proof}

We start by recalling the definition of a class of labeled trees $\Gamma$ due to Chen and Sohn~\cite{Chen}. The label of a vertex $v$ will be called its {\em status} and denoted by $sta(v)$. The labels in the definition are needed in order to be able to describe the family. But since in principle we are only interested in unlabeled graphs, we may later consider $\Gamma$ (as well as its subclasses) also as a class of unlabeled trees.

\begin{definition}
\label{def1}
Let $\Gamma$ be the family of labeled trees $T = T_k$ that can be obtained as follows. Let $T_0$ be a $P_6$ in which the two leaves have status $C$, the two support vertices have status $A$ and the remaining two vertices
have status $B$. If $k\geq 1$, then $T_k$ can be obtained recursively from
$T_{k-1}$ by one of the following operations, cf.~Fig.~\ref{fig:gamma}.
\begin{enumerate}[$\bullet$]
\item Operation $o_{1}$. For any $y\in V(T_{k-1})$,
if $sta(y)= C$ and $y$ is a leaf of $T_{k-1}$,
then add a path $xwvz$ and edge $xy$. Set $sta(x) =
sta(w) = B$, $sta(v) = A$, and $sta(z) = C$.
\item Operation $o_{2}$. For any $y\in V(T_{k-1})$,
if $sta(y)= B$, then add a path $xwv$ and edge $xy$.
Set $sta(x)=B$, $sta(w) = A$, and $sta(v) = C$.
\end{enumerate}
\end{definition}

\begin{figure}[ht!]
\begin{center}
\includegraphics*[width=15cm]{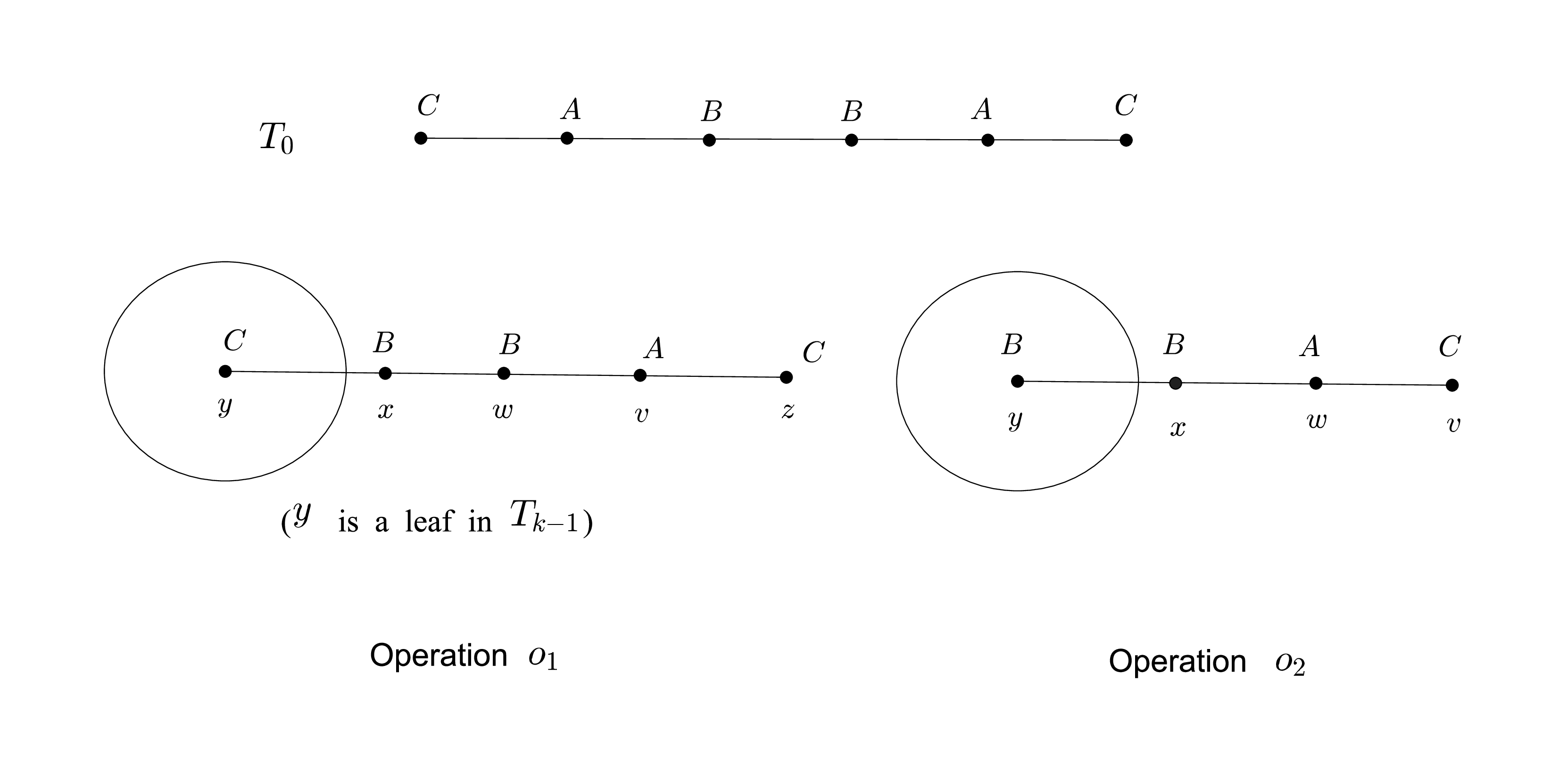}
\end{center}
\vspace*{-1.5cm}
\caption{The labelled tree $T_0$, and the two operations that define $\Gamma$}
\label{fig:gamma}
\end{figure}

The class of labeled trees $\Gamma$ is thus the smallest class of labeled trees that contains $T_0$ and can be built from it by successive applications of operations $o_{1}$ and $o_{2}$. Desormeaux, Haynes, and Henning~\cite{Desormeaux} proved that a tree $T$ of order at least $3$ satisfies $\gamma_{t}(T) = a(T)+1$ if and only if $T$ belongs to $\Gamma$ (considered as unlabeled trees). Since $P_2$ also satisfies this equality, we thus have:

\begin{theorem}[\cite{Desormeaux}]\label{t1}
If $T$ is a tree, then $\gamma_{t}(T)\leq a(T)+1$. Moreover, the equality holds if and only if $T\in \Gamma \cup \{P_2\}$.
\end{theorem}

If $G$ is a quasi-tree and $x\in V(G)$ such that $G-x$ is a tree, then we say that $x$ is a \emph{quasi-vertex}. We now partition quasi-trees into two classes as follows.

\begin{definition}\label{def2}
A quasi-tree $G$ is {\em type-1} if it contains a quasi-vertex $x$, such that $G-x\in \Gamma$. Otherwise, $G$ is {\em type-2}.
\end{definition}

Let $\Gamma_1\subset \Gamma$ be the class of trees from $\Gamma$ for which in their construction, at the last step operation $o_1$ was performed. Similarly, $\Gamma_2\subset \Gamma$ is be the class of trees from $\Gamma$ for which in their construction, at the last step operation $o_2$ was performed.  Moreover, if $T\in \Gamma_1$, then the vertices $x, w, v, z$ will be the vertices added in the last step (cf.~Fig.~\ref{fig:gamma}), and if $T\in \Gamma_2$, then the vertices $x, w, v$ will be the vertices added in the last step (cf.~Fig.~\ref{fig:gamma}). With this agreement we define the following six subclasses of type-1 quasi-trees; see Fig.~\ref{fig:1-2} for their schematic presentation.

\begin{enumerate}[$\bullet$]
\item $QT_{1}$ contains quasi-trees $G$  obtained from a tree $T\in \Gamma_1$ and an isolated vertex $h$ by adding $t\geq 2$ edges between $h$ and $V(T)\setminus \{x, w, v, z\}$.
\item $QT_{2}$ contains quasi-trees $G$  obtained from a tree $T\in \Gamma_2$ and an isolated vertex $h$ by adding $t\geq 2$ edges between $h$ and $V(T)\setminus \{x, w, v\}$.
\item $QT_{3}$ contains quasi-trees $G$ obtained from a tree $T\in \Gamma_1$ and an isolated vertex $h$ by adding $t\geq 2$ edges between $h$ and $\{x, w, v, z\}$.
\item $QT_{4}$ contains quasi-trees $G$ obtained from a tree $T\in \Gamma_2$ and an isolated vertex $h$ by adding $t\geq 2$ edges between $h$ and $\{x, w, v\}$.
\item $QT_{5}$ contains quasi-trees $G$  obtained from a tree $T\in \Gamma_1$ and an isolated vertex $h$ by adding $t_{1}\geq 1$ edges between $h$ and $\{x, w, v, z\}$, and $t_{2}\geq 1$ edges between $h$ and $V(T)\setminus \{x, w, v, z\}$.
\item $QT_{6}$ contains quasi-trees $G$  obtained from a tree $T\in \Gamma_2$ and an isolated vertex $h$ by adding $t_{1}\geq 1$ edges between $h$ and $\{x, w, v\}$, and $t_{2}\geq 1$ edges between $h$ and $V(T)\setminus \{x, w, v\}$.
\end{enumerate}

\begin{figure}[ht!]
\begin{center}
\includegraphics*[width=13cm]{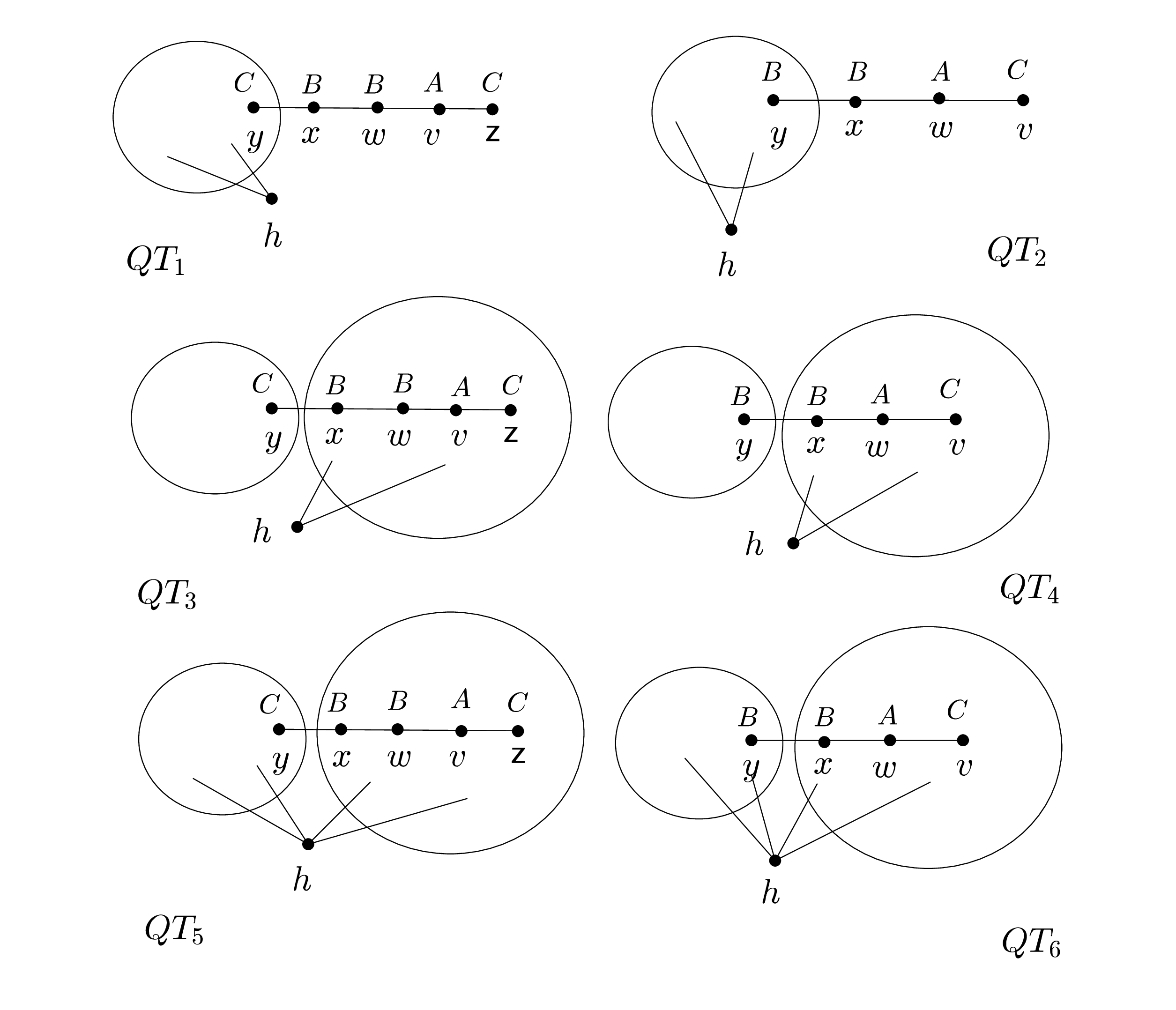}
\end{center}
\begin{center}
\end{center}
\vspace*{-1.5cm}
\caption{Type-1 quasi-trees from classes $QT_{1}, \ldots, QT_{6}$}
\label{fig:1-2}
\end{figure}

In the following series of six lemmas we respectively consider type-1 quasi-trees from classes $QT_{1}, \ldots, QT_{6}$.

\begin{lemma}\label{lem1}
Let $G\in QT_{1}$ and let
$G^{'}=G[V(G)\setminus \{x,\,w,\,v,\,z\}]$.
If  $\gamma_{t}(G^{'})\leq a(G^{'})+1$, then
 $\gamma_{t}(G)\leq a(G)+1$.
\end{lemma}

\begin{proof}
Let $D^{'}$ be a $\gamma_{t}$-set of $G^{'}$.
Then $D=D^{'}\cup \{w,\,v\}$ is a total dominating
set of $G$. So, $\gamma_{t}(G)\leq |D^{'}\cup \{w,\,v\}|= |D^{'}|+2=\gamma_{t}(G^{'})+2$. Next, we prove that
$a(G)\geq a(G^{'})+2$.

Clearly, $m(G) = m(G')+4$. Let $S$ be an optimal annihilation set of $G^{'}$. Then $\sum(S,\,G^{'}) \leq m(G')$. If $y\not\in S$, then $\sum(S\cup \{v,\,z\},\,G) = \sum(S,\,G^{'})+3 \leq m(G') + 3 < m(G)$. Thus, if $y\not\in S$, then $a(G)\geq |S\cup \{v,\,z\}|=|S|+2=a(G^{'})+2$. If $y\in S$, then we set $S^{*}=(S\setminus \{y\})\cup \{w,\,v,\,z\}$. As $d_{G^{'}}(y)\geq 1$, we have $\sum(S^{*},\,G) = \sum(S,\,G^{'})-d_{G^{'}}(y)+d_{G}(w)+d_{G}(v)+d_{G}(z) \leq \sum(S,\,G^{'})-1+5 \leq m(G') + 4 = m(G)$. Thus, also if $y\in S$, we have the conclusion $a(G)\geq |S^{*}|=|S|+2=a(G^{'})+2$.

In conclusion, if $\gamma_{t}(G^{'})\leq a(G^{'})+1$, then $\gamma_{t}(G) \leq \gamma_{t}(G^{'})+2\leq a(G^{'})+3\leq a(G)+1$.
\end{proof}

\begin{lemma}\label{lem2}
Let $G\in QT_{2}$ and let $G^{'}=G[V(G)\setminus \{x,\,w,\,v\}]$.
If  $\gamma_{t}(G^{'})\leq a(G^{'})+1$, then
 $\gamma_{t}(G)\leq a(G)+1$.
\end{lemma}

\begin{proof}
Let $D^{'}$ be a $\gamma_{t}$-set of $G^{'}$.
Then $D=D^{'}\cup \{x,\,w\}$ is a total dominating
set of $G$. So, $\gamma_{t}(G)\leq |D^{'}\cup \{x,\,w\}|= |D^{'}|+2=\gamma_{t}(G^{'})+2$. Next, we prove that
$a(G)\geq a(G^{'})+2$.

Clearly, $m(G) = m(G')+3$. Let $S$ be an optimal annihilation set of $G^{'}$. Then $\sum(S,\,G^{'}) \leq m(G')$. If $y\not\in S$, then $\sum(S\cup \{w,\,v\},\,G)=\sum(S,\,G^{'})+3 \leq m(G')+3=m(G)$. Thus, $a(G)\geq |S\cup \{w,\,v\}|=|S|+2=a(G^{'})+2$. If $y\in S$, then we set $S^{*}=(S\setminus \{y\})\cup \{x,\,w,\,v\}$. Since $d_{G^{'}}(y)\geq 2$, we have $\sum(S^{*},\,G) = \sum(S,\,G^{'})-d_{G^{'}}(y)+d_{G}(x)+d_{G}(w)+d_{G}(v) =
\sum(S,\,G^{'})-d_{G^{'}}(y)+5 \leq m(G') -2+5 = m(G)$. Thus, $a(G)\geq |S^{*}|=|S|+2=a(G^{'})+2$.

Hence, when $\gamma_{t}(G^{'})\leq a(G^{'})+1$,
we  have $\gamma_{t}(G)\leq
\gamma_{t}(G^{'})+2\leq a(G^{'})+3\leq a(G)+1$.
\end{proof}

\begin{lemma}\label{lem3}
If $G\in QT_{3}$, then $\gamma_{t}(G)\leq a(G)+1$.
\end{lemma}

\begin{proof}
Let $G^{'}=G[V(G)\setminus \{h,\,x,\,w,\,v,\,z\}]$. Then $G^{'} \in \Gamma$. By Theorem \ref{t1}, we have $\gamma_{t}(G^{'})=a(G^{'})+1$. Let $D^{'}$ be a $\gamma_{t}$-set of $G^{'}$ and let $d_{G}(h)=t$. Then  $m(G) = m(G') + t + 4$, where $2\leq t\leq 4$.

If $t=2$, then $G$ is a unicyclic graph. So $G$ is a cactus graph and by the validity of Conjecture~\ref{con1} for cacti due to  Bujt\'{a}s and Jakovac~\cite{Bujtas} we have $\gamma_{t}(G)\leq a(G)+1$. Assume in the rest that $t\in \{3,4\}$. Since $D=D^{'}\cup \{w,\,v\}$ is a total dominating set of $G$, we have $\gamma_{t}(G)\leq |D^{'}\cup \{w,\,v\}| = |D^{'}|+2=\gamma_{t}(G^{'})+2$. It remains to prove that $a(G)\geq a(G^{'})+2$.

Let $S$ be an optimal annihilation set of $G^{'}$. Then $\sum(S,\,G^{'}) \leq m(G')$. Since  $y$ is a pendent
vertex in $G^{'}$, we must have $y\in S$ by the definition of optimal annihilation sets. Let $S^{*}=(S\setminus \{y\})\cup \{z,\,w,\,v\}$. Since $d_{G^{'}}(y)=1$, we have $\sum(S^{*},\,G) = \sum(S,\,G^{'})-d_{G^{'}}(y)+d_{G}(z)+d_{G}(w)+d_{G}(v) = \sum(S,\,G^{'})-1+5+t \leq m(G')+4+t = m(G)$. Thus $a(G)\geq |S^{*}|=|S|+2=a(G^{'})+2$ and we are done.
\end{proof}

\begin{lemma}\label{lem4}
If $G\in QT_{4}$, then $\gamma_{t}(G)\leq a(G)+1$.
\end{lemma}

\begin{proof}
Let $G^{'}=G[V(G)\setminus \{h,\,x,\,w,\,v\}]$. Then $G^{'}\in \Gamma$. By Theorem \ref{t1} we have
$\gamma_{t}(G^{'})=a(G^{'})+1.$
Let $D^{'}$ be a $\gamma_{t}$-set of $G^{'}$.
Since $t\geq 2$, $D=D^{'}\cup \{x,\,w\}$ is a total dominating
set of $G$. So, $\gamma_{t}(G)\leq |D^{'}\cup \{x,\,w\}|= |D^{'}|+2=\gamma_{t}(G^{'})+2$. Hence the lemma will be proved by showing that $a(G)\geq a(G^{'})+2$.

If $d_{G}(h)=t$, then $m(G) = m(G')+t+3$, where $t\in \{2,3\}$. Let $S$ be an optimal annihilation set of $G^{'}$. Then $\sum(S,\,G^{'}) \leq m(G')$. If $y\not\in S$, then $\sum(S\cup \{w,\,v\},\,G)=\sum(S,\,G^{'}) + d_{G}(w)+d_{G}(v)\leq m(G')+3+t = m(G)$. Thus, $a(G)\geq |S\cup \{w,\,v\}|=|S|+2=a(G^{'})+2$. If $y\in S$, then set $S^{*}=(S\setminus \{y\})\cup \{x,\,w,\,v\}$. Since $d_{G^{'}}(y)\geq 2$, we have $\sum(S^{*},\,G) = \sum(S,\,G^{'})-d_{G^{'}}(y)+d_{G}(x)+d_{G}(w)+d_{G}(v) = \sum(S,\,G^{'})-d_{G^{'}}(y)+5+t \leq m(G')-2+5+t = m(G)$. Thus, $a(G)\geq |S^{*}|=|S|+2=a(G^{'})+2$.
\end{proof}

\begin{lemma}\label{lem5}
Let $G\in QT_{5}$ and let
$G^{'}=G[V(G)\setminus \{x,\,w,\,v,\,z\}]$.
If  $\gamma_{t}(G^{'})\leq a(G^{'})+1$, then
 $\gamma_{t}(G)\leq a(G)+1$.
\end{lemma}

\begin{proof}
Let $D^{'}$ be a $\gamma_{t}$-set of $G^{'}$.
Then $D=D^{'}\cup \{w,\,v\}$ is a total dominating
set of $G$. So, $\gamma_{t}(G)\leq |D^{'}\cup \{w,\,v\}|= |D^{'}|+2=\gamma_{t}(G^{'})+2$.

Let $d_{G}(h)=t$ and recall that $t = t_1 + t_2$ by the definition of $ QT_{5}$. Then $m(G) = m(G') +t_{1}+4$, where $t_{1} \in [t-1]$. Let $S$ be an optimal annihilation set of $G^{'}$. Then $\sum(S,\,G^{'}) \leq m(G')$.

If $h\not\in S$ and $y\not\in S$, then $\sum(S\cup \{v,\,z\},\,G)=\sum(S,\,G^{'}) + d_{G}(v)+d_{G}(z)\leq m(G')+3+t_{1} < m(G)$. Thus, $a(G)\geq |S\cup \{v,\,z\}|=|S|+2=a(G^{'})+2$. If $h\not\in S$ and $y\in S$, then set $S^{*}=(S\setminus \{y\})\cup \{z,\,w,\,v\}$. As $d_{G^{'}}(y)\geq 1$, we have $\sum(S^{*},\,G) = \sum(S,\,G^{'})-d_{G^{'}}(y)+d_{G}(z)+d_{G}(w)+d_{G}(v)\leq \sum(S,\,G^{'})-1+5+t_{1}\leq m(G')+4+t_{1} = m(G)$. Thus, $a(G)\geq |S^{*}|=|S|+2=a(G^{'})+2$. By our assumption that $\gamma_{t}(G^{'})\leq a(G^{'})+1$, $\gamma_{t}(G)\leq a(G)+1$.

Now, we consider the case when $h\in S$.

If $t_{2}\geq 2$, since $d_{G^{'}}(h)=t_{2}$ and $y$ may belong to $S$, then $\sum((S\setminus\{h\})\cup \{w,\,v,\,z\},\,G)\leq \sum(S,\,G^{'})-d_{G^{'}}(h)+1 + d_{G}(w)+ d_{G}(v)+d_{G}(z)\leq m(G')-2+1+5+t_{1}=m(G)$. Thus, $a(G)\geq |(S\setminus\{h\})\cup \{w,\,v,\,z\}|=|S|+2=a(G^{'})+2$. By our assumption that $\gamma_{t}(G^{'})\leq a(G^{'})+1$, $\gamma_{t}(G)\leq a(G)+1$.

Now, we suppose that $t_{2}=1$. If $t_{1}=1$, then $G$ is a unicyclic graph. So $G$ is a cactus graph and by the validity of Conjecture 1.1 for cacti due to  Bujt\'{a}s and Jakovac~\cite{Bujtas} we have $\gamma_{t}(G)\leq a(G)+1$. If $t_{1}=2$, then $G$ is a $C$-disjoint graph, since it can be obtained by planting trees to vertices of a  theta graph. Hence, by the validity of Conjecture 1.1 for $C$-disjoint graphs due to  Yue et al.~\cite{Yue},  we have $\gamma_{t}(G)\leq a(G)+1$. So, we may assume that $t_{1}\geq 3$.

Let $G^{''}=G-\{h,\,x,\,w,\,v,\,z\}$. Then $m(G'')+5+t_{1}=m(G)$. Let $S^{''}$ be an optimal annihilation set of $G^{''}$. Then $\sum(S^{''},\,G^{''}) \leq m(G'')$. Obviously, $G^{''}\in\Gamma$. By Theorem 2.2, we have
$\gamma_{t}(G^{''})=a(G^{''})+1.$ Let $D^{''}$ be a $\gamma_{t}$-set of $G^{''}$. Since $t_{1}\geq 3$, then $D^{''}\cup \{w,\,v\}$ is a total dominating set of $G$.
Then $\gamma_{t}(G)\leq |D^{''}\cup \{w,\,v\}|=\gamma_{t}(G^{''})+2$.
Now, we prove that $a(G)\geq a(G^{''})+2$.
Since $y$ and the unique neighbor of $h$ in $G^{''}$ may belong to $S^{''}$, we have $\sum(S^{''}\cup \{v,\,z\},\,G)\leq \sum(S^{''},\,G^{''})+2+d_{G}(v)+d_{G}(z)\leq m(G'')+2+t_{1}+3=m(G)$. Thus, $a(G)\geq |S^{''}\cup \{v,\,z\}|=|S^{''}|+2=a(G^{''})+2$. Therefore, $\gamma_{t}(G)\leq a(G)+1$.
\end{proof}

\begin{lemma}\label{lem6}
Let $G\in QT_{6}$. Then
 $\gamma_{t}(G)\leq a(G)+1$.
\end{lemma}

\begin{proof}
Let $d_{G}(h)=t$ and recall that $t = t_1 + t_2$ by the definition of $ QT_{6}$. Let $G^{'}=G-\{h,\,x,\,w,\,v\}$. Then $G^{'}\in\Gamma$. By Theorem \ref{t1}, we have
$\gamma_{t}(G^{'})=a(G^{'})+1.$ Let $D^{'}$ be a $\gamma_{t}$-set of $G^{'}$. Then either $D=D^{'}\cup \{x,\,w\}$ or $D=D^{'}\cup \{w,\,v\}$ is a total dominating set of $G$. So, $\gamma_{t}(G)\leq |D^{'}\cup \{x,\,w\}| = |D^{'}|+2=\gamma_{t}(G^{'})+2$ or $\gamma_{t}(G)\leq |D^{'}\cup \{w,\,v\}| = |D^{'}|+2=\gamma_{t}(G^{'})+2$. To complete the argument we next again prove that $a(G)\geq a(G^{'})+2$.

Obviously, $m(G)= m(G')+t_{1}+t_{2}+3$, where $t_{1},\,t_{2}\in [t-1]$. Let $S$ be an optimal annihilation set of $G^{'}$. Then $\sum(S,\,G^{'}) \leq m(G')$.

If $y\not\in S$, then $\sum(S\cup \{w,\,v\},\,G)\leq\sum(S,\,G^{'})+ t_{2}+d_{G}(w)+d_{G}(v) \leq m(G')+t_{2}+t_{1}+3= m(G)$. Thus, $a(G)\geq |S\cup \{w,\,v\}|=|S|+2=a(G^{'})+2$. If $y\in S$, since $d_{G^{'}}(y)\geq 2$, then $\sum((S\setminus \{y\})\cup \{x,\,w,\,v\},\,G)\leq\sum(S,\,G^{'})-d_{G^{'}}(y)+ t_{2}+d_{G}(x)+d_{G}(w)+d_{G}(v) \leq m(G')-2+t_{2}+t_{1}+5= m(G')+t_{1}+t_{2}+3= m(G)$. Thus, $a(G)\geq |(S\setminus \{y\})\cup \{x,\,w,\,v\}|=|S|+2=a(G^{'})+2$, and we are done.
\end{proof}

With the above lemmas in hand we are now in a position to prove Theorem~\ref{t3}. If $G$ is itself a tree, then the result follows from Theorem~\ref{t1}. Assume hence that $G$ contains at least one cycle. If $n(G)=3$, then
$G = C_3$ for which $\gamma_{t}(C_{3})=2=a(C_{3})+1$ holds. We may thus assume in the rest that $n(G)\geq 4$. We consider the following two cases.

\medskip\noindent
{\bf Case 1}: $G$ is a type-$2$ quasi-tree.\\
Let $h$ be a quasi-vertex of $G$ and let $d_{G}(h)=t$. Since $G$ has at least one cycle, we have $t\geq 2$. Since $G$ is a type-$2$ quasi-tree, $G-h\not\in \Gamma$. Let $S$ be an optimal annihilation set of $G-h$. As
$$\sum(S,\,G)\leq\sum(S,\,G-h)+t\leq m(G-h)+t=m(G)\,,$$
we have $a(G)\geq |S|=a(G-h)$. We further consider the following two subcases.

\medskip\noindent
{\bf Case 1.1}: There exists a $\gamma_{t}$-set $D$ of $G-h$  such that $N_{G}(h)\cap D\neq \emptyset$.\\
In this case, $D$ is also a  total dominating set of $G$.
So, $\gamma_{t}(G)\leq |D|=\gamma_{t}(G-h)$.
Since $G-h$ is a tree, by Theorem \ref{t1},
we have $$\gamma_{t}(G)\leq
\gamma_{t}(G-h)\leq a(G-h)+1\leq a(G)+1.$$

\medskip\noindent
{\bf Case 1.2}: For each $\gamma_{t}$-set $D$ of $G-h$ we have $N_{G}(h)\cap D=\emptyset$. \\
Let $\mathcal{S}(G-h)$ and $\mathcal{L}(G-h)$  be the set of support vertices and the set of pendent vertices of $G-h$, respectively.  Since $G-h$ is a tree, each vertex $u\in \mathcal{S}(G-h)$ belongs to every $\gamma_{t}$-set of $G-h$. By our assumption that $N_{G}(h)\cap D=\emptyset$ for each $\gamma_{t}$-set $D$ of $G-h$, we have $\mathcal{S}(G-h)\cap N_{G}(h)=\emptyset$. We now claim that $\gamma_{t}(G)\leq \gamma_{t}(G-h)+1$.

First, assume that there exists a pendent vertex $x$ in $\mathcal{L}(G-h)$ such that $x\in  N_{G}(h)$. Let $D$ be a $\gamma_{t}$-set of $G-h$. Then $D\cup\{x\}$ is a total dominating set of $G$, as the unique neighbor of $x$ in $G-h$ is a support vertex belonging to $D$. So, $\gamma_{t}(G)\leq |D\cup\{x\}|=\gamma_{t}(G-h)+1$. Second, assume that  for any pendent vertex $x$ of $G-h$ we have $x\not\in  N_{G}(h)$. Then  there exists a vertex $y$ in $V(G-h)\setminus (\mathcal{S}(G-h)\cup\mathcal{L}(G-h))$  such that $y\in  N_{G}(h)$, as $t\geq 2$.  Since  $D$ is a  $\gamma_{t}$-set of $G-h$,  there exists a neighbor of $y$, say $z$, such that $z\in D$. Then $D\cup\{y\}$ is a total dominating set of $G$. Hence $\gamma_{t}(G)\leq |D\cup\{y\}|=\gamma_{t}(G-h)+1$ and the claim is proved.

Since $G-h\not\in \Gamma$ and $G-h\ncong P_{2}$, Theorem \ref{t1} implies that $\gamma_{t}(G-h)\leq a(G-h)$. So, $$\gamma_{t}(G)\leq
\gamma_{t}(G-h)+1\leq a(G-h)+1\leq a(G)+1\,,$$
which completes the argument for type-2 quasi-trees.

\medskip\noindent
{\bf Case 2}: $G$ is a type-1 quasi-tree.\\
Let $h$ be a quasi-vertex of $G$, such that $G-h\in \Gamma$. To prove that $\gamma_{t}(G)\leq  a(G)+1$ we use induction on $f(G) = n(G)+ m(G)+ n_{1}(G)$, where $n_1(G)$ is the number of leaves of $G$.

Let $d_{G}(h)=t$. Since $G$ has at least one cycle, $t\geq 2$. Since $G-h\in \Gamma$, we have $f(G) = n(G)+ m(G)+ n_{1}(G)\geq (1+6)+(5+t)+0\geq (1+6)+(5+2)+0=14$ with equality only if $n(G)=m(G)=7$ and $n_{1}(G)=0$. (Recall that a quasi-tree $G$ is connected.) So the base case of our induction is $f(G)=14$, in which case  we have $G\cong C_{7}$. As $\gamma_{t}(C_{7})=4=3+1=a(C_{7})+1$, the desired result holds for the base case.

Suppose now that $G$ is a type-1 quasi-tree with $f(G)\geq 15$. Since $G-h\in \Gamma$, we must have
$G\in QT_{i}$ for some $i \in [6]$.

If $G\in QT_{1}$, set  $G^{'}=G[V(G)\setminus \{x,\,w,\,v,\,z\}]$. Then $G^{'}$ is a type-1 quasi-tree, as  $G^{'}-h \in \Gamma$. Obviously, $f(G^{'})<f(G)$. So, by the induction  hypothesis,  $\gamma_{t}(G^{'})\leq  a(G^{'})+1$. Thus, by Lemma~\ref{lem1}, we have  $\gamma_{t}(G)\leq  a(G)+1$.

If $G\in QT_{2}$, we set $G^{'} = G[V(G)\setminus \{x,\,w,\,v\}]$. Then $G^{'}$ is a type-1 quasi-tree as $G^{'}-h \in \Gamma$. Clearly, $f(G^{'})<f(G)$ and hence by the induction  hypothesis, $\gamma_{t}(G^{'})\leq  a(G^{'})+1$. Lemma~\ref{lem2} implies that $\gamma_{t}(G)\leq  a(G)+1$.

If $G\in QT_{3}$, then $\gamma_{t}(G)\leq  a(G)+1$ holds by Lemma \ref{lem3}, and if $G\in QT_{4}$, then the same conclusion follows from Lemma~\ref{lem4}.

If $G\in QT_{5}$, set $G^{'}=G[V(G)\setminus \{x,\,w,\,v,\,z\}]$. Then $G^{'}$ is a tree or a type-1 quasi-tree with $f(G^{'})<f(G)$. If $G^{'}$ is a tree, then $\gamma_{t}(G^{'})\leq a(G^{'})+1$ by Theorem~\ref{t1} and consequently $\gamma_{t}(G)\leq  a(G)+1$ by Lemma~\ref{lem5}. If $G^{'}$ is a type-1 quasi-tree, then since $f(G^{'})<f(G)$, the induction  hypothesis yields $\gamma_{t}(G^{'})\leq  a(G^{'})+1$. Thus $\gamma_{t}(G)\leq  a(G)+1$ by Lemma~\ref{lem5}.

Finally, if $G\in QT_{6}$, then $\gamma_{t}(G)\leq  a(G)+1$ by Lemma~\ref{lem6}.

We have completed the argument for Case 2 which completes the proof of Theorem~\ref{t3}.

\section{Composition graphs}
\label{sec:composite}

In this section, we prove that Conjecture \ref{con1} holds for some composition graphs, the first four of which can have minimum degree equal to $2$, and the last of which can also have minimum degree $1$.

\subsection{Triangulations of graphs}

The {\em triangulation}, $\tau(G)$, of a graph $G$, is the graph obtained from $G$ by adding, for each edge $e = uv$ of $G$, a new vertex $x_e$ and the two edges $x_eu$ and $x_ev$,  cf.~\cite{Rosenfeld}.

\begin{proposition}
If $G$ is a connected graph, then $\gamma_{t}(\tau(G))\leq a(\tau(G))+1$.
\end{proposition}

\begin{proof}
Note first that $m(\tau(G)) = 3m(G)$. Let $S=V(\tau(G))\setminus V(G)$. Then $\sum(S,\,\tau(G)) = 2m(G) < 3m(G) = m(\tau(G))$. Thus, $a(\tau(G))\geq |S| = m(G)$. Let $D$ be a vertex subset composed of arbitrary  $n(G)-1$ vertices of $G$. Since $G$ is connected, we infer that $D$ is a total dominating set of $\tau(G)$. So, $\gamma_{t}(\tau(G))\leq |D| = n(G)-1 \leq m(G) < a(\tau(G))+1$.
\end{proof}

\subsection{Double graphs}

The {\em double graph}, $G^{*}$, of a graph $G$ is constructed as follows. Let $G_1$ and $G_2$ be disjoint  copies of $G$, where for every $u\in V(G)$ its copy in $G_i$, $i\in [2]$, is denoted by $u_i$. Then $G^{*}$ is obtained from the disjoint union of $G_1$ and $G_2$ by adding, for each edge $uv$ of $G$, the edges $u_{1}v_{2}$ and $u_{2}v_{1}$, cf.~\cite{Munarini}.

\begin{proposition}
If $G$ is a connected graph with $\gamma_{t}(G)\leq a(G)+1$, then $\gamma_{t}(G^{*})\leq a(G^{*})+1$.
\end{proposition}

\begin{proof}
Clearly, $m(G^{*}) = 4m(G)$. Let $S$ be an optimal annihilation set of $G$. Then $\sum(S,\,G)\leq m(G)$. So, $\sum(S,\,G^{*})=2\sum(S,\,G) \leq 2m(G) < 4m(G)$. Thus, $a(G^{*})\geq |S|=a(G)$.

Let $D$ be a $\gamma_{t}$-set of $G$. Then it is straightforward to see that the copy of $D$ in $G_1$ (or in $G_2$ for that matter) is a total dominating set of $G^{*}$. It follows that $\gamma_{t}(G^{*}) \leq |D| =\gamma_{t}(G)$. Hence, if  $\gamma_{t}(G)\leq a(G)+1$, then $\gamma_{t}(G^{*})\leq
\gamma_{t}(G)\leq a(G)+1\leq a(G^{*})+1$.
\end{proof}

\subsection{Bijection graphs}

Let $G$ and $H$ be disjoint graphs with $n(G) = n(H)$ and let $f : V(G) \rightarrow V(H)$ be a bijection. The {\em bijection graph} $B(G,H,f)$ is obtained from the disjoint union of $G$ and $H$ by adding the edges $uf(u)$, $u\in V(G)$, cf.~\cite{Shcherbak}. If $G\cong H$, then the bijection graph $B(G, H ,f)$ is also known as  {\em permutation graph}.

\begin{proposition}
If $G$ and $H$ are connected graphs with $n(G) = n(H)$ and $f:V(G)\rightarrow V(H)$ is a bijection, then $\gamma_{t}(B(G,H,f)) \leq a(B(G,H,f)) + 1$.
\end{proposition}

\begin{proof}
Let $G$, $H$ and $f$ be as stated, and set $B = B(G,H,f)$. We may without loss of generality assume that $m(G) \geq m(H)$ (otherwise consider $B(G,H,f^{-1})$). Set $S = V(H)$ and note that $\sum(S, B) = 2m(H) + n(G) \leq m(G) + m(H) + n(G) = m(B)$. Thus, $a(B)\geq |S| = n(H)$. Further, since $H$ is connected, $V(H)$ is a total dominating set of  $B$. Thus,  $\gamma_{t}(B) \leq n(H) < n(H) +1 \leq  a(B) +1$.
\end{proof}

\subsection{The Mycielskian}

The famous construction of Mycielski from~\cite{Mycielski}, which is especially important in chromatic graph theory, can be described as follows. The {\em Mycielski graph} $\mu(G)$  of a graph $G$ contains $G$ itself as an isomorphic subgraph, together with $n+1$ additional vertices: to each vertex $v_{i}$ of $G$, a vertex $u_{i}$ is added, and there is another vertex $w$. The vertex $w$ is adjacent to all the vertices $u_{i}$, and each edge $v_{i}v_{j}$ of $G$ yields edges $u_{i}v_{j}$ and $v_{i}u_{j}$.

Next, we prove that if a graph $G$ satisfies Conjecture~\ref{con1}, then so does $\mu(G)$. For other results on different kinds of domination in Mycielski graphs see~\cite{chen-2006, kwon-2021}.

\begin{proposition}
If $G$  is a connected graph with $\gamma_{t}(G)\leq a(G)+1$, then
$
\gamma_{t}(\mu(G))\leq a(\mu(G))+1.
$
 \end{proposition}

 \begin{proof}
Let $S$ be an optimal annihilation set in $G$. Then $\sum(S,\,G)\leq m(G)$, and therefore $\sum(S,\,\mu(G))=2\sum(S,\,G)\leq 2m(G)$. Then $\sum(S\cup\{w\},\,\mu(G)) = \sum(S,\,\mu(G)) + d_{\mu(G)}(w) \leq 2m(G) + n(G) < 3m(G)+n(G) = m(\mu(G))$. Hence $a(\mu(G))\geq a(G)+1$. On the other hand, from~\cite{chen-2006} we know that $\gamma_{t}(\mu(G)) = \gamma_{t}(G) + 1$. Using these facts together with the assumption $\gamma_{t}(G)\leq a(G)+1$ we get
\begin{align*}
\gamma_{t}(\mu(G)) & = \gamma_{t}(G) + 1 \le a(G) + 2 \le (a(\mu(G)) - 1) + 2 = a(\mu(G)) + 1
\end{align*}
and we are done.
\end{proof}
\subsection{Universally-identifying graphs}

A \emph{universal vertex} of a graph $G$ is a vertex adjacent to all other vertices of $V (G)$. For a connected graph $G$ with $v\in V(G)$ and another graph $H$ containing a universal vertex, we define a new graph, named \textit{universally-identifying graph}, denoted by $G_v\ast H$, which is obtained by identifying the vertex $v$ of $G$ with the universal vertex of $H$.

In this subsection, we prove that Conjecture \ref{con1} holds for universally-identifying graphs.

\begin{proposition}\label{uni}
Let $G$ and $H$ be connected graphs. If $H$ is a graph with $n(H)\geq \lceil\frac{n(G)}{3}+2\rceil$ and a universal vertex, and $v\in V(G)$, then $\gamma_t(G_v\ast H)\leq a(G_v\ast H)+1$.
\end{proposition}

\begin{proof}
If $D$ is a $\gamma_t$-set of $G$, then $D\cup\{v\}$ is a total dominating set of $G_v\ast H$. Hence $\gamma_t(G_v\ast H)\leq \gamma_t(G)+1$. Since $\gamma_t(G)\leq \frac{2n(G)}{3}$ holds for any connected graph $G$, see~\cite{henning-2007}, we have
  $$\gamma_t(G_v\ast H)\leq \gamma_t(G)+1\leq \frac{2n(G)}{3}+1\leq \left\lfloor\frac{n(H)+n(G)-1}{2}\right\rfloor+1\leq a(G_v\ast H)+1$$
and we are done.
\end{proof}

If $G$ in Proposition \ref{uni} has minimum degree $1$ or $2$, then $G_v\ast H$ may also have minimum degree $1$ or $2$, so long as $G-v$ has a vertex in $V(G)\setminus \{v\}$ of minimum degree $1$ or $2$.

In our final result we provide another class of graphs, each of its members has minimum degree $1$ or $2$ and satisfies Conjecture~\ref{con1}.

\begin{proposition}
Let $G$ be a connected graph with a cut-vertex $v\in V(G)$ such that $d_{G}(v)=k\geq \left\lfloor\frac{n(G)}{4}+4\right\rfloor$. If $G-v$ has a component of order $n(G)-k$ and this component contains exactly one neighbor of $v$, then $\gamma_t(G)\leq a(G)+1$.
\end{proposition}

\begin{proof}
Set $n = n(G)$, and let $N_G(v)=\{v_1,\ldots,v_k\}$. By our assumption, there is a vertex, say $v_1\in N_G(v)$, such that the component $G_1$ of $G-v$ containing $v_1$ has $n(G_1)=n-k$. Let $V_0=\{v_2,v_3,\ldots,v_k\}$, $V'=V(G_1)\cup \{v\}$, $G'_{v}=G[V']$, and $G_0=G[V_{0}\cup\{v\}]$. Then $v$ is a universal vertex of $G_0$, and $G=G'_v\ast G_0$ can be seen as an universally-identifying graph. Since $k\geq \left\lfloor\frac{n}{4}+4\right\rfloor$, we have $\frac{n}{2}\geq \frac{2(n-k+1)}{3}+1$. By a similar reasoning as that in the proof of Proposition \ref{uni}, we have $$\gamma_t(G)\leq\gamma_t(G')+1\leq \frac{2(n-k+1)}{3}+1\leq \lfloor\frac{n}{2}\rfloor+1\leq a(G)+1\,.$$
\end{proof}

\vspace*{1cm}
\section*{Acknowledgements}

Hongbo Hua  is supported by National Natural Science Foundation
of China under Grant No. 11971011. Sandi Klav\v{z}ar acknowledges the financial support from the Slovenian Research Agency (research core funding P1-0297 and projects N1-0095, J1-1693, J1-2452).

\vspace*{1cm}
\noindent{\bf Declaration of competing interest}

The authors declare that they have no known competing financial interests or personal relationships that could have appeared to influence the work reported in this paper.

\end{document}